\documentclass[11pt,letterpaper]{amsart}
\usepackage{amssymb}
\usepackage{color}
\usepackage[linkcolor=blue,citecolor=blue, colorlinks=true,backref=true,bookmarks=true]{hyperref}

\theoremstyle{plain}
\newtheorem{theorem}{Theorem}[section]
\newtheorem{corollary}[theorem]{Corollary}

\newtheorem{lemma}[theorem]{Lemma}
\newtheorem{proposition}[theorem]{Proposition}
\newtheorem{definition}[theorem]{Definition}

\theoremstyle{definition}
\newtheorem{remark}[theorem]{Remark}

\def\R{\mathbb{R}}

\def\H{\mathbb{H}}

\def\C{\mathbb{C}}

\def\N{\mathbb{N}}

\def\H{\mathbb{H}}

\numberwithin{equation}{section}

\title{Sharp weighted CR trace Sobolev inequalities}
\author{Gunhee Cho}
\address{Department of Mathematics\\
	University of California, Santa Barbara\\
	552 University Rd, Isla Vista, CA 93117.}
\email{gunhee.cho@math.ucsb.edu}

\author{Zetian Yan}
\address{Department of Mathematics \\ Penn State University \\ University Park \\ PA 16802 \\ USA}
\email{zxy5156@psu.edu}

\date{\today}

\begin{document}
	\begin{abstract}
		We establish a sharp Sobolev trace inequality on the Siegel domain $\Omega_{n+1}$ involving the weighted norm-$W^{2,2}(\Omega_{n+1}, \rho^{1-2[\gamma]})$. The inequality is closely related the realization of fractional powers of the sub-Laplacian on $\H^n=\partial \Omega_{n+1}$ as generalized Dirichlet-to-Neumann operators associated to the weighted poly-sublaplacian, generalizing observations of Frank--Gonz{\'a}lez--Monticelli--Tan.
	\end{abstract}

	\maketitle
	
	\section{Introduction}
 In \cite{Case20}, Case recovered the fractional Laplacian $(-\overline{\Delta})^{\gamma}$ on $\R^n$ as a Dirichlet-to-Neumann operator associated to the {\textit{weighted poly-Laplacian}} $L_{2k}:=\Delta_m^k$ determined by $\gamma\in (0,\infty)\backslash \N$, where $\Delta_m:=\Delta+my^{-1}\partial_y$ is the weighted Laplacian on $\R_+^{n+1}:=\R^n\times (0,\infty)$ for $m=1-2[\gamma]$ and $k:=\lfloor \gamma\rfloor+1$. Specifically, given functions $f^{(2j)}\in C^{\infty}(\R^n)\cap H^{\gamma-2j}(\R^n)$, $j\in[0,\lfloor \gamma/2\rfloor]$ and $\phi^{(2j)}\in C^{\infty}(\R^n)\cap H^{\lfloor \gamma\rfloor-[\gamma]-2j}(\R^n)$, $j\in[0, \lfloor \gamma\rfloor-\lfloor \gamma/2\rfloor-1]$, if $V$ is a solution of 
        \begin{equation}\label{classicex}
		\left\{
		\begin{array}{ll}
		L_{2k}V=0, &  ~~\mbox{in}~~ \R^{n+1}_+, \\
		B_{2j}^{2\gamma}(V)=f^{(2j)}, & ~~\mbox{for}~~ j\in[0,\lfloor \gamma/2\rfloor],\\
		B_{2[\gamma]+2j}^{2\gamma}(V)=\phi^{(2j)}, &  ~~\mbox{for}~~ j\in[0, \lfloor \gamma\rfloor-\lfloor \gamma/2\rfloor-1],
		\end{array}
		\right.
	\end{equation}
there exists constants $c_{\gamma,j}$ and $d_{\gamma,j}$ such that
\begin{equation}\label{pickup}
    \begin{split}
        B_{2\gamma-2j}^{2\gamma}(V)=c_{\gamma,j} (-\overline{\Delta})^{\gamma-2j} f^{(2j)},  & ~~\mbox{for}~~ j\in[0,\lfloor \gamma/2\rfloor],\\
        B_{2\lfloor \gamma\rfloor-2j}^{2\gamma}(V)=d_{\gamma,j} (-\overline{\Delta})^{\lfloor \gamma\rfloor-[\gamma]-2j}\phi^{(2j)}, &  ~~\mbox{for}~~ j\in[0, \lfloor \gamma\rfloor-\lfloor \gamma/2\rfloor-1],
    \end{split}
\end{equation}
where $B_{2j}^{2\gamma}$ and $B_{2[\gamma]+2j}^{2\gamma}$ are boundary operators constructed in \cite[Definition 3.1]{Case20}.

Applying the Dirichlet principle, one deduces that
\begin{equation}\label{sharpSobolev}
\begin{split}
    \mathcal{E}_{2\gamma}(U)\geqslant & \sum_{j=0}^{\lfloor \gamma/2\rfloor} c_{\gamma,j} \oint_{\R^n}f^{(2j)} (-\overline{\Delta})^{\gamma-2j} f^{(2j)} dx \\
    +& \sum_{j=0}^{\lfloor \gamma\rfloor-\lfloor \gamma/2\rfloor-1} d_{\gamma,j} \oint_{\R^n}\phi^{(2j)} (-\overline{\Delta})^{\lfloor \gamma\rfloor-[\gamma]-2j}\phi^{(2j)} dx, 
    \end{split}
\end{equation}
for any $U\in \mathcal{C}^{2\gamma}\cap W^{k,2}(\R^{n+1}_+, y^{1-2[\gamma]})$, where the Dirichlet energy $\mathcal{E}_{2\gamma}(U)$ is given by
\begin{equation*}
\begin{split}
    \mathcal{E}_{2\gamma}(U):=\int_{\R^{n+1}_+}U L_{2k}U y^{m}dxdy&+\sum_{j=0}^{\lfloor \gamma/2\rfloor} \oint_{\R^n}  B_{2j}^{2\gamma}(U)B_{2\gamma-2j}^{2\gamma}(U) dx\\
    &-\sum_{j=0}^{\lfloor \gamma\rfloor-\lfloor \gamma/2\rfloor-1}\oint_{\R^n}  B_{2[\gamma]+2j}^{2\gamma}(U)B_{2\lfloor \gamma\rfloor-2j}^{2\gamma}(U) dx;
    \end{split}
\end{equation*}
see \cite{Case20} for more details. We may regard (\ref{sharpSobolev}) as a functional inequality for the Sobolev trace embedding
\begin{equation*}
    W^{k,2}(\R^{n+1}_+, y^{1-2[\gamma]})\hookrightarrow \bigoplus_{j=0}^{\lfloor \gamma/2\rfloor}H^{\gamma-2j}(\R^n)\oplus \bigoplus_{j=0}^{\lfloor \gamma\rfloor-\lfloor \gamma/2\rfloor-1} H^{\lfloor \gamma\rfloor-[\gamma]-2j}(\R^n).
\end{equation*}
Combining (\ref{sharpSobolev}) with Lieb's sharp fractional Sobolev inequality \cite{Lieb83} yields the sharp Sobolev trace inequality for the embedding
\begin{equation*}
    W^{k,2}(\R^{n+1}_+, y^{1-2[\gamma]})\hookrightarrow \bigoplus_{j=0}^{\lfloor \gamma/2\rfloor}L^{\frac{2n}{n-2\gamma+4j}}(\R^n)\oplus \bigoplus_{j=0}^{\lfloor \gamma\rfloor-\lfloor \gamma/2\rfloor-1} L^{\frac{2n}{n-2\lfloor \gamma\rfloor+2[\gamma]+4j}}(\R^n)
\end{equation*}
when $n>2\gamma$.
 
The sharp version of (\ref{sharpSobolev}) for $\gamma\in (0,1)$ stems from Caffarelli and Silvestre's seminal paper \cite{CS07}. Furthermore, the special case $\gamma=\frac{3}{2}$ was proven by Ache--Chang \cite{AC17}; the special case $\gamma=\N_0+\frac{1}{2}$ was prove by Yang \cite{Y19}; the general case $\gamma\in (0,\infty)\backslash \N$ was proven by Case \cite{Case20}. 

The remarkable breakthrough in \cite{Case20} is the construction of boundary operators associated to the weighted poly-Laplacian. They are defined recursively in terms of the weighted Laplacian $\Delta_m$ and the weighted derivative $y^m\partial_y$ in $\R^{n+1}_+$, and the induced Laplacian $\overline{\Delta}$ on $\R^n$ as follows:
\begin{equation*}
\begin{split}
    \iota^* \circ \Delta_m^j&=\sum_{l=0}^{j}(-1)^l \binom{j}{l}\frac{(\lfloor \gamma\rfloor-l)!\Gamma(\gamma-j-l)}{(\lfloor \gamma\rfloor-j)!\Gamma(\gamma-2l)} \overline{\Delta}^{j-l}B_{2l}^{2\gamma},\\
    \iota^* \circ y^m\partial_y\Delta_m^j&=(-1)^{j+1}\sum_{l=0}^{j} \binom{j}{l}\frac{(\lfloor \gamma\rfloor-l)!\Gamma(1+2l-\lfloor \gamma\rfloor+[\gamma])}{(\lfloor \gamma\rfloor-j)!\Gamma(1+j+l-\lfloor \gamma\rfloor+[\gamma])} \overline{\Delta}^{j-l}B_{2[\gamma]+2l}^{2\gamma},
    \end{split}
\end{equation*}
where $\iota^*: C^{\infty}(\overline{\R^{n+1}_+})\to C^{\infty}(\R^n)$ is the restriction operator. These definitions are justified by three properties. First, they are such that the associated Dirichlet form $\mathcal{Q}_{2\gamma}(U,V):=\frac{1}{2}\left(\mathcal{E}_{2\gamma}(U+V)-\mathcal{E}_{2\gamma}(U)-\mathcal{E}_{2\gamma}(V)\right)$ is symmetric. Second, the boundary operators are covariant with respect to the group of conformal isometries of $(\R^{n+1}_+;\R^n)$; see \cite[Section 4]{Case20} for the precise statement. Third, as shown in (\ref{pickup}), they are such that the generalized Dirichlet-to-Neumann operators recover the fractional Laplacians. 

In the CR setting, fractional covariant operators of order $2\gamma$ for $\gamma\in \R$, was defined from scattering theory on a K{\"a}hler--Einstein manifold $\mathcal{X}$ \cite{GS08,HPT08,EMM91}. They are pseudodifferential operators whose principal symbol agress with the pure fractional powers of the CR sub-Laplacian $(-\Delta_b)^{\gamma}$ on the boundary $M=\partial \mathcal{X}$. In the special case of the Heisenberg group $\H^n$, they are simply the intertwining operators on the CR sphere calculated in \cite{BFM07} using classical representation theory. 

In the same spirit of Caffarelli and Silvestre, Frank, Gonz{\'a}lez, Monticelli and Tan formulated an extension problem on $\Omega_{n+1}$ and constructed the pure fractional powers of the sub-Laplacian on $\H^n$ as the Dirichlet-to-Neumann operator of a degenerate elliptic equation using the Fourier transform and the spectral resolution of the operator. Moreover, we have the following sharp Sobolev trace inequality.

\begin{theorem}[\cite{FGMT15}]\label{1.1}
    Given $\gamma\in (0,1)$, for any $U \in \mathcal{C}_{D}^{2 \gamma}$, $\phi\in S^{\gamma,2}(\H^n)$, we have
    \begin{equation}\label{CRSobolev}
        \mathcal{E}_{2\gamma}(U)\geqslant 2^{1-2\gamma}\gamma \frac{\Gamma(1-\gamma)}{\Gamma(1+\gamma)}\oint_{\H^n} \phi P_{\gamma}^{\theta}\phi dzdt., \quad \phi=\iota^* U,
    \end{equation}
where $\iota^*: C^{\infty}(\overline{\Omega}_{n+1})\to C^{\infty}(\H^n)$ is the restriction operator, the corresponding function space $\mathcal{C}_{D}^{2 \gamma}$ is defined in Theorem \ref{Dirichlet-Principle} and the Dirichlet energy $\mathcal{E}_{2 \gamma}(U)$ is given by
\begin{equation*}
    \mathcal{E}_{2\gamma}(U)=\int_{\Omega_{n+1}}\left(|\partial_{\rho}U|^2+\frac{1}{2}\sum_{j=1}^n|X_j(U)|^2+|Y_j(U)|^2+\rho^2|\partial_tU|^2 \right)\rho^{1-2\gamma}dzdtd\rho.
\end{equation*}
 Equality is attained if and only if $U$ is the unique solution for the extension problem  
\begin{equation}
    \left\{
	\begin{array}{ll}
		\left(\partial^2_\rho+(1-2\gamma)\rho^{-1}\partial_{\rho}+\rho^2 \partial_t^2+\Delta_b\right)U=0, &  ~~\mbox{in}~~ \Omega_{n+1},\\
	U=\phi, & ~~\mbox{on}~~ \H^{n}.
		\end{array}
		\right.
\end{equation} 
\end{theorem}
Our main result gives the sharp Sobolev trace inequality for $\gamma\in (1,2)$.

\begin{theorem}\label{main}
   Given $\gamma\in (1,2)$, for any $U \in \mathcal{C}_{D}^{2 \gamma}$, $\phi\in S^{\gamma,2}(\H^n)$, we have
\begin{equation}\label{sharp2}
    \mathcal{E}_{2 \gamma}(U) \geq  2^{3-2\gamma}\frac{\Gamma(2-\gamma)}{\Gamma(\gamma)}\oint_{\H^{n}} \phi P^{\theta}_{\gamma}\phi dzdt,\quad \phi=\iota^* U,
\end{equation}
where $\iota^*: C^{\infty}(\overline{\Omega}_{n+1})\to C^{\infty}(\H^n)$ is the restriction operator, the corresponding function space $\mathcal{C}_{D}^{2 \gamma}$ is defined in Theorem \ref{Dirichlet-Principle}. The Dirichlet energy $\mathcal{E}_{2 \gamma}(U)$ in (\ref{sharp2}) is given by
\begin{equation*}
\begin{split}
    \mathcal{E}_{2\gamma}(U):=&\int_{\Omega_{n+1}}UL_{4}U \rho^{1-2[\gamma]}dzdtd\rho \\
	&+ \oint_{\H^n} B_{0}^{2\gamma}(U) B_{2\gamma}^{2\gamma}(U)dzdt-\oint_{\H^n} B_{2[\gamma]}^{2\gamma}(U) B_{2}^{2\gamma}(U)dzdt,
 \end{split}
\end{equation*}
where the associated boundary operators are
    \begin{equation}
        \begin{split}
        B_{0}^{2\gamma}=\iota^*, \quad & B_{2}^{2\gamma}=- \iota^* \circ \mathcal{T}+\frac{1-[\gamma]}{[\gamma]}\Delta_b\circ B_{0}^{2\gamma}, \\ B_{2[\gamma]}^{2\gamma}=-\iota^* \circ \rho^{1-2[\gamma]}\partial_{\rho}, \quad  &B_{2\gamma}^{2\gamma}= \iota^* \circ \rho^{1-2[\gamma]}\partial_{\rho} \mathcal{T}-\frac{1+[\gamma]}{[\gamma]}\Delta_b\circ B_{2[\gamma]}^{2\gamma},
         \end{split}
    \end{equation}
if we set
\begin{equation*}
	 \mathcal{T}:=\partial^2_\rho+(1-2[\gamma])\partial_\rho+\rho^2\partial^2_t.
\end{equation*}
Equality in (\ref{sharp2}) is attained if and only if $U$ is the unique solution for the extension problem 
\begin{equation}\label{1.8}
    \left\{
	\begin{array}{ll}
		L_{4}V=0, &  ~~\mbox{on}~~ \Omega_{n+1}, \\
		B_{0}^{2\gamma}(V)=\phi, & ~~\mbox{on}~~ \H^{n},\\
		B_{2[\gamma]}^{2\gamma}(V)=0, &  ~~\mbox{on}~~ \H^{n},
	\end{array}
	\right.
\end{equation}
where
\begin{equation*}
\begin{split}
    L_4&:= L_{[\gamma]}^2+T^2,\\
     L_{[\gamma]}=\partial^2_\rho+(1-2[\gamma])&\rho^{-1}\partial_{\rho}+\rho^2 \partial_t^2+\Delta_b, \quad T=2\frac{\partial}{\partial t}.
     \end{split}
\end{equation*}
In particular, $P^{\theta}_{\gamma}$ can be recovered from the solution of (\ref{1.8}) by 
\begin{equation*}
    P^{\theta}_{\gamma}\phi=2^{-3+2\gamma}\frac{\Gamma(\gamma)}{\Gamma(2-\gamma)} \lim_{\rho\to 0^+}\rho^{1-2[\gamma]}\partial_{\rho} L_{[\gamma]}U.
\end{equation*}   
\end{theorem}

The article is organized as follows:

In Section 2 we review some basic concepts on the Heisenberg group $\H^n$ as a CR manifold and as the boundary of Siegel domain. 

In Section 3 we recall the identification of fractional powers of the sub-Laplacian on $\H^n$ via scattering theory for the hyperbolic Laplacian on the Siegel domain $\Omega_{n+1}$ \cite{HPT08}. In the same spirit in \cite{Case20}, we define the weighted poly-sublaplacian on $\Omega_{n+1}$ and prove a nice factorization in Lemma \ref{factorization}. We also give a direct relationship between the weighted poly-sublaplacian on $\Omega_{n+1}$ and the complex hyperbolic poly-Poisson operator in hyperbolic space. 

In Section 4 we introduce the boundary operators associated to the weighted poly-sublaplacian for $\gamma\in (0,2)$ and study their role in recovering fractional powers of the sub-Laplacian as generalized Dirichlet-to-Neumann operators. The key results here link our boundary operators to the asymptotics of solutions to a Poisson equation relevant to scattering theory \cite{HPT08} and show that the Dirichlet form determined by our boundary operators is symmetric. 

Due to the extra term $\rho^2\partial^2_t$ in extension problems, it is hard to construct higher order boundary operators which satisfy the `mutual exclusion' property; see Remark \ref{obs} for a precise statement. 

In Section 5 we prove the existence and uniqueness of solutions to extension problems (\ref{Dirichlet1}) and (\ref{Dirichlet2}), which are degenerate elliptic boundary value problems. This enables us to prove the main result of this section, Theorem \ref{Caffarelli-Silvestre-extension}.

In Section 6 we prove various sharp trace inequalities. First, Theorem \ref{Dirichlet-Principle} asserts a Dirichlet principle for solutions of the Dirichlet problem (\ref{Dirichlet1}) and (\ref{Dirichlet2}). As a result, we recover Theorem \ref{1.1} in Corollary \ref{cover} and prove the sharp Sobolev trace inequality (\ref{sharp2}) as a direct consequence of Corollary \ref{cover2}. Combining Corollary \ref{cover} and Corollary \ref{cover2} with the sharp Sobolev inequalities on $\H^n$ \cite{FL10} we obtain a sharp Sobolev trace inequality stated in Corollary \ref{6.4}.

{\bf{Acknowledgements.}} We would like to thank Prof. Jeffrey S. Case for useful discussion.

	\section{Prelimiaries}
	\subsection{The Heisenberg group as a CR manifold}
	The $n$-dimensional Heisenberg group $\H^n$ is defined as the set $\C^n\times \R$ endowed with the group law
	\begin{equation*}
		\left(\hat{z},\hat{t}\right)\cdot \left(z,t\right)=\left(\hat{z}+z,\hat{t}+t+2\mathrm{Im}\langle\hat{z},z\rangle_{\C^n}\right)
	\end{equation*}
	for $\left(\hat{z},\hat{t}\right), \left({z},{t}\right)\in \H^n$, where $\langle\cdot,\cdot\rangle_{\C^n}$ is the standard inner product in $\C^n$. The CR structure on the Heisenberg group $\H^n$ is inherited from the embedding $\H^n\hookrightarrow \C^{n+1}$. The holomorphic tangent bundle $T^{1,0}\H^n$ is given by
	\begin{equation*}
		T^{1,0}\H^n=T^{1,0}\C^{n+1}\cap \C T\H^n=\mathrm{span}_{\C}\langle Z_1,\cdots, Z_n\rangle,
	\end{equation*}
 where 
        \begin{equation}
		Z_j=\frac{\partial}{\partial z_j}+i\bar{z}_j\frac{\partial}{\partial t}, \quad \quad \bar{Z}_j=\frac{\partial}{\partial \bar{z}_j}-iz_j\frac{\partial}{\partial t}, \quad j=1,\cdots,n.
	\end{equation}

 We also set
	\begin{equation}
		X_j=\frac{\partial}{\partial x_j}+2y_j\frac{\partial}{\partial t}, \quad Y_j=\frac{\partial}{\partial y_j}-2x_j\frac{\partial}{\partial t}, \quad j=1,\cdots,n, \quad T=2\frac{\partial}{\partial t},
	\end{equation}
	which form a base of the Lie algebra of vector fields on $\H^n$ and are left invariant with respect to the group action. We immediately observe that the associated maximal complex distribution $H(\H^n)=\mathrm{Re}\left\{T^{1,0}\H^n\oplus T^{0,1}\H^n\right\}$ is simply 
	\begin{equation*}
		H(\H^n)=\mathrm{span}_{\C}\langle X_1,\cdots, X_n, Y_1,\cdots, Y_n \rangle,
	\end{equation*}
	which carries the complex structure $J_b: H(\H^n)\to H(\H^n)$ defined by
	\begin{equation*}
		J_b X_j=Y_j, \quad J_b Y_j=-X_j, \quad j=1,\cdots,n.
	\end{equation*}
	The real tangent bundle $T\H^n$ satisfies
	\begin{equation}\label{span}
		T\H^n=H(\H^n)\oplus \R T.
	\end{equation}

	The canonical contact form $\theta$ on $\H^n$ is precisely
	\begin{equation}\label{contact}
		\theta=\frac{1}{2}\left[dt+i\sum_{j=1}^n\left(z_jd\bar{z}_j-\bar{z}_jdz_j\right)\right],
	\end{equation}
	and the vector field $T$ is the characteristic direction satisfying
	\begin{equation*}
		\theta(T)=1,\quad \quad T\rfloor d\theta=0.
	\end{equation*}
	The Levi form $L_{\theta}$ associated to $\theta$ is given by
	\begin{equation*}
		L_{\theta}=-id\theta=\sum_{j=1}^n dz_j\wedge d\bar{z}_j,
	\end{equation*}
	which is a positive definite matrix. This tells us that $(\H^n, \theta)$ is strictly pseudoconvex.

 
	
	The sub-Laplacian associated to $\theta$ is defined as
	\begin{equation}\label{sublaplacian}
		\Delta_b u:=\sum_{j=1}^n\left[\bar{Z}_jZ_j+Z_j\bar{Z}_j\right]u=\frac{1}{2}\sum_{j=1}^n\left[X^2_j+Y^2_j\right]u;
	\end{equation}
	the differential operator is linear, second order, degenerate elliptic, and it is hypoelliptic being the sum of squares of smooth vector fields satisfying the H\"ormander condition. 
	
	\subsection{The Heisenberg group as the boundary of the Siegel domain}
	The Heisenberg group $\H^n$ may be identified as the boundary of a domain in $\C^{n+1}$. Indeed, let $\Omega_{n+1}$ be the Siegel domain, defined by
	\begin{equation}
		\Omega_{n+1}:=\left\{\zeta=(z,z_{n+1})\in \C^n\times \C| q(\zeta)>0\right\},
	\end{equation}
	where 
	\begin{equation*}
		q(\zeta)=\mathrm{Im}z_{n+1}-|z|^2.
	\end{equation*} 
	The boundary $\partial \Omega_{n+1}=\left\{\zeta\in \C^{n+1}|q(\zeta)=0\right\}$ is an oriented CR manifold of hypersurface type with the CR structure induced by $\C^{n+1}$, which is CR equivalent to $\H^n$ by the CR isomorphism $\mathcal{G}:\H^n\to \partial \Omega_{n+1}$,
	\begin{equation*}
		\begin{split}
			\mathcal{G}(z,t)&=\left(z,t+i|z|^2\right), \quad (z,t)\in\H^n,\\ \mathcal{G}^{-1}(\zeta)&=(z,\mathrm{Re}z_{n+1}), \quad \zeta\in \partial \Omega_{n+1}. 
		\end{split}  
	\end{equation*}
	
	The boundary manifold $\partial \Omega_{n+1}$ inherits a natural CR structure from the complex structure of the ambient manifold $\C^{n+1}$, $-q$ as a defining function. The pseudohermitian structure on $\H^n$ is then given (via pullback) by the contact form
	\begin{equation*}
		\theta=\frac{i}{2}\left(\bar{\partial}-\partial\right)(-q)=\frac{1}{4}(dz_{n+1}+d\bar{z}_{n+1})+\frac{i}{2}\sum_{j=1}^n\left(z_jd\bar{z}_j-\bar{z}_jdz_j\right),
	\end{equation*}
	which agrees with the definition given in (\ref{contact}).
	
	In order to complete a basis for $T\Omega_{n+1}$, we define 
	\begin{equation*}
		\xi_{n+1}=\frac{1}{2}(N-iT), \quad \quad \bar{\xi}_{n+1}=\frac{1}{2}(N+iT),
	\end{equation*}
	where
	\begin{equation*}
		T=2\frac{\partial}{\partial t}, \quad \quad N=-2\frac{\partial}{\partial q}.
	\end{equation*}
	Then a frame in $\Omega_{n+1}$ is given by $\left\{Z_j, \bar{Z}_j,\xi_{n+1},\bar{\xi}_{n+1} \right\}_{j=1}^{n}$, and the dual coframe of this basis is $\left\{\theta^j,\theta^{\bar{j}},\theta, dq\right\}$ with $dq=\partial q+\bar{\partial}q$. In particular, the dual coframe in $\C T\H^n$ of $\left\{Z_j, \bar{Z}_j,T \right\}$ is simply $\left\{\theta^j,\theta^{\bar{j}},\theta\right\}$.
	
	On the other hand, the functions $t=\mathrm{Re}Z_{n+1}\in \R$, $z_1,\cdots,z_n\in \C^n$ and $\rho=(2q)^{1/2}\in \R_+$ give coordinates in $\Omega_{n+1}\simeq\H^n\times \R_+$. For the defining function $-q$, we can construct a K\"{a}hler form in $\Omega_{n+1}$ as 
	\begin{equation*}
		\omega_+=-\frac{i}{2}\partial \bar{\partial} \log(q)=\frac{i}{2}\left(-\frac{\partial \bar{\partial}q}{q}+\frac{\partial q\wedge \bar{\partial}q}{q^2}\right).
	\end{equation*}
	By direct calculation, we observe that
	\begin{equation}
		-\partial \bar{\partial}q=\sum_{j=1}^n dz_j\wedge d\bar{z}_j,\quad \quad \partial q\otimes \bar{\partial}q+ \bar{\partial} q\otimes \partial q=\frac{1}{2}\left[(dq)^2+4\theta^2\right].
	\end{equation}
	Then after the change $q=\frac{\rho^2}{2}$, the Riemannian metric $g^+$ associated to $\omega_+$ is given by
	\begin{equation*}
		g^+=\frac{1}{2}\left(\frac{d\rho^2}{\rho^2}+\frac{2\delta_{\alpha\bar{\beta}}\theta^\alpha\otimes\theta^{\bar{\beta}}}{\rho^2}+\frac{4\theta}{\rho^2}\right).
	\end{equation*}
	Then $(\Omega_{n+1}, \omega_+)$ is a K\"{a}hler manifold with constant holomorphic curvature. We recall from \cite{FGMT15} that the Laplace operator on the K\"{a}hler manifold $(\Omega_{n+1}, \omega_+)$ is calculated as
	\begin{equation*}
		\Delta_+ u=\mathrm{Tr}\left(i\partial \bar{\partial}u\right),
	\end{equation*}
	and the following formula decomposes the Laplacian into tangential and normal components relative to the level sets of $q$:
	\begin{equation}\label{Laplacian}
		\Delta_+ =q\left[q(\partial_q^2+\partial_t^2)+\frac{1}{2}\Delta_b-n\partial_q\right],
	\end{equation}
	with $\Delta_b$ the sub-Laplacian (\ref{sublaplacian}). The spectrum of the Laplacian $-\Delta_+$ consists of an absolutely continuous part
	\begin{equation*}
		\sigma_{ac}(-\Delta_+)=\left[\frac{(n+1)^2}{4},\infty \right),
	\end{equation*}
	and the pure point spectrum satisfying
	\begin{equation*}
		\sigma_{pp}(-\Delta_+)\subset \left(0,\frac{(n+1)^2}{4}\right),
	\end{equation*}
	and moreover it consists of a finite set of $L^2$-eigenvalues.

	\section{Scattering theory on the Siegel domain}
	
	Given $\gamma\in (0,m)\backslash \N$, $m=n+1$, for each $f\in C^{\infty}(\H^n)\cap S^{\gamma,2}(\H^n)$, where $S^{\gamma,2}(\H^n)$ is the Folland--Stein space characterized in \cite{BFM07}, there exists a unique solution $u_s$ of the Poisson equation
	\begin{equation}\label{poisson}
		\Delta_+u_s +s(m-s)u_s=0, \quad s=\frac{m+\gamma}{2}.
	\end{equation}
	Then the Poisson map is defined as $\mathcal{P}(s):C^{\infty}(\H^n)\to C^{\infty}(\overline{\Omega}_{n+1})$ by $f\to u_s$; see \cite{HPT08} for a Poisson kernel for (\ref{poisson}). For our purposes, it suffices to know that there are functions $F,G\in C^{\infty}(\overline{\Omega}_{n+1})$ such that
	\begin{equation}\label{expansion}
		\left\{
		\begin{array}{lr}
			u_s=\mathcal{P}(s)(f)=q^{m-s}F+q^s G, &  \\
			F|_{\H^n}=f.&
		\end{array}
		\right.
	\end{equation}
	Define the scattering operator $\mathcal{S}(s)(f):=G|_{\H^n}$. The CR {\textit{fractional GJMS operator}} on $(\H^n,\theta)$ of order $2\gamma$ is given by
	\begin{equation}\label{GJMS}
		P_{\gamma}^{\theta}:=2^{\gamma}\frac{\Gamma(\gamma)}{\Gamma(-\gamma)} \mathcal{S}(\frac{m+\gamma}{2}),
	\end{equation}
	which satisfies
	\begin{equation}
		P_{\gamma}^{e^{2w}\theta}f=e^{-(n+1+\gamma)w}P_{\gamma}^\theta \left(e^{(n+1-\gamma)w}f\right), \quad w\in C^{\infty}(\H^n).
	\end{equation}
	It is known from \cite{FGMT15} that 
	\begin{equation}
		P_{\gamma}^{\theta}=\left(2 |T|\right)^{\gamma} \frac{\Gamma\left( \frac{1+\gamma}{2}+\frac{-\Delta_b}{2|T|}\right)}{\Gamma\left(\frac{1-\gamma}{2}+\frac{-\Delta_b}{2|T|}\right)}, \quad \gamma\in (0,1).
	\end{equation}
	
	The function $F$ in (\ref{expansion}) is determined modulo $O(q^{\infty})$ by $f$ by finding the Taylor series solution to
	\begin{equation}\label{F}
		\left\{
		\begin{array}{lr}
			\left(\Delta_+ +\frac{m^2-\gamma^2}{4}\right)\left(q^{\frac{m-\gamma}{2}}F\right)=O(q^{\infty}), &  \\
			F|_{\H^n}=f.&
		\end{array}
		\right.
	\end{equation}
	Similarly, the function $G$ in (\ref{expansion}) is determined modulo $O(q^{\infty})$ by $P_{\gamma}^\theta f$ by finding the Taylor series solution to
	\begin{equation}\label{G}
		\left\{
		\begin{array}{lr}
			\left(\Delta_+ +\frac{m^2-\gamma^2}{4}\right)\left(q^{\frac{m+\gamma}{2}}G\right)=O(q^{\infty}), &  \\
			G|_{\H^n}=\mathcal{S}(\frac{m+\gamma}{2})(f).&
		\end{array}
		\right.
	\end{equation}
	For this reason, we would like to understand the formal solution of 
	\begin{equation}\label{formal}
		\left(\Delta_+ +s(m-s)\right)(V)=0, \quad s=\frac{m\pm \gamma}{2}, \quad \gamma\in (0,m)\backslash \N
	\end{equation} 
	Because of a computational error in \cite[Lemma 3.4]{HPT08}, we recast the calculation for the correct expansion.
	
	\begin{proposition}\label{formalep}
		Suppose that $u_s\in C^{\infty}(\overline{\Omega}_{n+1})$ satisfying $u_s=q^{m-s}F+q^s G$ is the solution of (\ref{expansion}). Then we have
		\begin{equation*}
			F=\sum_{l=0}^{\infty}q^lf_l,\quad \quad f_l=(-1)^l\frac{\Gamma(m-2s+1)}{l!\Gamma(m-2s+l+1)}P^s_l(f_0),
		\end{equation*}
		and 
		\begin{equation*}
			G=\sum_{l=0}^{\infty}q^lg_l, \quad \quad g_l=(-1)^l\frac{\Gamma(2s-m+1)}{l!\Gamma(2s-m+l+1)}G^s_l(g_0),
		\end{equation*}
		where $P^s_l$ and $G^s_l$ are differential operators of order $2l$ defined by (\ref{definitionP}) and (\ref{definitionG}), respectively, with leading symbol
		\begin{equation}\label{symbol}
			\sigma(P^s_l)=\sigma(G^s_l)=\sigma(\Delta_b^l).
		\end{equation}
		In particular, we have $P^s_l=G^{m-s}_l$.
	\end{proposition}
	\begin{proof}
		We think of (\ref{Laplacian}) as a variable-coefficient differential operator with respect to vector field $q\frac{\partial}{\partial q}$ and vector fields tangent to $\H^n$, i.e.,
		\begin{equation}
			\Delta_+=\sum_{i=0}^2 q^i L_i,
		\end{equation}
		where
		\begin{equation}
			L_0=\left(q\frac{\partial}{\partial q}\right)^2-(n+1) q\frac{\partial}{\partial q}, \quad \quad L_1=\frac{\Delta_b}{2}, \quad \quad L_2=\frac{T^2}{4}.
		\end{equation}
		We substitute $V=q^{m-s}F$ into (\ref{formal}) and obtain that
		\begin{equation}\label{expansion2}
			\begin{split}
				&\left(\Delta_+ +s(m-s)\right)(q^{m-s}F)=\left(\sum_{i=0}^2 q^i L_i +s(m-s)\right)(q^{m-s}F)\\
				&=\left(L_0 +s(m-s)\right)(q^{m-s}F)+q L_1(q^{m-s}F)+q^2 L_2(q^{m-s}F).
			\end{split}
		\end{equation}
		We assume that $F$ has the formal expansion $F=\sum_{l=0}^{\infty}q^lf_l$. Direct calculation shows that
		\begin{equation}\label{L_0}
			\begin{split}
				\left(L_0 +s(m-s)\right)&(q^{m-s}F)=\left(m-2s+1\right)q^{m-s+1}\frac{\partial F}{\partial q}+q^{m-s+2}\frac{\partial^2 F}{\partial q^2}\\
				&=\sum_{l=1}^{\infty}l\left(m-2s+1\right)q^{m-s+l}f_l+\sum_{l=1}^{\infty}l\left(l-1\right)q^{m-s+l}f_l\\
				&=\sum_{l=1}^{\infty}l\left(m-2s+l\right)q^{m-s+l}f_l.
			\end{split}
		\end{equation}
		Plugging (\ref{L_0}) into (\ref{expansion2}) yields
		\begin{equation}\label{recursive}
			f_{-1}=0, \quad f_0=f, \quad f_l=-\frac{L_1 f_{l-1}+L_2 f_{l-2}}{l\left(m-2s+l\right)},\quad l\geqslant 1.
		\end{equation}
		We assume that 
		\begin{equation}\label{P_l}
			f_l=(-1)^l\frac{\Gamma(m-2s+1)}{l!\Gamma(m-2s+l+1)}P^s_l(f),
		\end{equation}
		where $P^s_l$ is an operator depending on $L_1, L_2$ and $s$.
		
		Combining this with (\ref{recursive}), we observe that
		\begin{equation}
			\begin{split}
				&P^s_{-1}=0,\quad P^s_0=\mathrm{id},\\
				P^s_l=L_1\circ P^s_{l-1}-&(l-1)(m-2s+l-1)L_2\circ P^s_{l-2}, \quad l\geqslant 1.
			\end{split}
		\end{equation}
		By induction one sees that $P^s_l$ is a homogeneous polynomial of degree $l$ in the commuting variables $L_1$ and $\sqrt{L_2}$. Hence, the recursion relation for $P^s_l$ reduces to the recursion for the polynomials $p_l^s$ of one variable $x$ defined by
		\begin{equation}\label{p}
			\begin{split}
				&p^s_{-1}=0,\quad p^s_0(x)=1,\\
				p^s_l=x p^s_{l-1}-&(l-1)(m-2s+l-1)p^s_{l-2}, \quad l\geqslant 1.
			\end{split}
		\end{equation}
		Obviously, $p^s_l$ is a polynomial of degree $l$ with the leading coefficient $1$.
		Moreover, $P^s_l$ is recovered from $p^s_l$ by 
		\begin{equation}\label{definitionP}
			P^s_l=\left(\sqrt{L_2}\right)^l \circ p^s_l\left(\frac{L_1}{\sqrt{L_2}}\right).
		\end{equation}
		Similarly, we have
		\begin{equation}\label{G_l}
			g_l=(-1)^l\frac{\Gamma(2s-m+1)}{l!\Gamma(2s-m+l+1)}G^s_l(\mathcal{S}(s)(f)),
		\end{equation}
		where $G^s_l$ is given by
		\begin{equation}\label{definitionG}
			G^s_l=\left(\sqrt{L_2}\right)^l \circ g^s_l\left(\frac{L_1}{\sqrt{L_2}}\right),
		\end{equation}
		and the polynomial $g^s_l$ is determined by
		\begin{equation}\label{g}
			\begin{split}
				&g^s_{-1}=0,\quad g^s_0(x)=1,\\
				g^s_l=x g^s_{l-1}-&(l-1)(2s-m+l-1)g^s_{l-2}, \quad l\geqslant 1,
			\end{split}
		\end{equation}
		which implies that $P^s_l=G^{m-s}_l$. Finally, (\ref{symbol}) follows from (\ref{definitionP}), (\ref{definitionG}) immediately.
	\end{proof}
\begin{remark}\label{unsolvable}	
The analogue of the operator $P^s_l$ on $\R^n$ is precisely $(-\overline{\Delta})^l$, independent of the value of $s$. However, from the recursion relation (\ref{p}), we observe that $P^s_l$ depends on the value of $s$ when $l\geqslant 2$ and it is hard to solve the recursion relation (\ref{p}) explicitly.  
\end{remark}
	There are two ways to study the CR fractional GJMS operators via an extension. The first approach is to identify solutions of (\ref{expansion}) as elements of the kernel of a second-order “weighted” Laplacian on the Siegel domain.
	
	\begin{lemma}\label{kernel}
		Let  $\gamma\in (0,m)\backslash \N$. We define
		\begin{equation}\label{L}
			L_{\gamma}=\partial^2_\rho+(1-2\gamma)\rho^{-1}\partial_{\rho}+\rho^2 \partial_t^2+\Delta_b.  
		\end{equation}
		Then
		\begin{equation*}
			L_\gamma\circ \rho^{-m+\gamma}\circ \mathcal{P}(\frac{m+\gamma}{2})=0.
		\end{equation*}
	\end{lemma}
	\begin{proof}
		After the change $q=\frac{\rho^2}{2}$, we have
		\begin{equation}\label{Laplacian2}
			\Delta_+=\frac{\rho^2}{4}\left(\partial^2_\rho-(2n+1)\rho^{-1}\partial_{\rho}+\rho^2 \partial_t^2+\Delta_b\right).
		\end{equation}
		A direct computation using (\ref{Laplacian2}) shows that
		\begin{equation}\label{transformation}
			\frac{1}{4}L_{\gamma}\left(\rho^{-m+\gamma}U\right)=\rho^{-m+\gamma-2}\left(\Delta_++\frac{m^2-\gamma^2}{4}\right)U,
		\end{equation}
		for all $U\in C^{\infty}(\overline{\Omega}_{n+1})$. The conclusion follows from the definition of the Poisson operator $ \mathcal{P}(\frac{m+\gamma}{2})$.
	\end{proof}
	
	The second approach to studying the CR fractional GJMS operators via an extension is to identify solutions of (\ref{expansion}) as elements of the kernel of a product of $L_\gamma$. This can be done as follows.
	
	First, let $\gamma\in (0,m)\backslash \N$. Set $k:=\lfloor \gamma\rfloor+1$. The {\textit{weighted poly-sublaplacian determined by}} $\gamma$ is
	\begin{equation}\label{weighted}
		L_{2k}:=\prod_{j=0}^{k-1}L_{\gamma-2j}=L_{\gamma-2k+2}\circ \cdots\circ L_\gamma,
	\end{equation}
	where $L_\mu$ is defined in (\ref{L}).
	
	Second, let $\gamma\in (0,m)\backslash \N$. Set $s:=\frac{m+\gamma}{2}$ and 
	\begin{equation*}
		D_s:=\Delta_+ +s(m-s).
	\end{equation*}
	The {\textit{complex hyperbolic poly-Poisson operator determined by}} $\gamma$ is
	\begin{equation}\label{hyperbolic}
		L^+_{2k}:=\prod_{j=0}^{k-1}D_{s-j}.
	\end{equation}
	
	The following two lemmas capture the essential features of this relationship as needed to study CR sharp Sobolev trace inequalities on the Siegel domain.
	
	First, to understand the boundary operators, it is useful to rewrite the factorization of $L_{2k}^+$.
	\begin{lemma}
		Let $\gamma\in (0,m)\backslash \N$ and $L_{2k}^+$ be the weighted poly-sublaplacian defined in (\ref{hyperbolic}). Then
		\begin{equation}
			L_{2k}^+=\prod_{j=0}^{\lfloor \gamma/2\rfloor} D_{s_j} \circ \prod_{j=0}^{\lfloor \gamma\rfloor-\lfloor \gamma/2\rfloor-1}D_{\tilde{s}_j},
		\end{equation}
		where $s_j=\frac{m+\gamma}{2}-j$ and $\tilde{s}_j=\frac{m+\lfloor \gamma\rfloor-[\gamma]}{2}-j$.
	\end{lemma}
	\begin{proof}
		Separating (\ref{hyperbolic}) into terms with $s-j>\frac{m}{2}$ and $s-j<\frac{m}{2}$, we compute that
		\begin{equation*}
			L_{2k}^+=\prod_{j=0}^{\lfloor \gamma/2\rfloor} D_{s_j} \circ \prod_{j=\lfloor \gamma/2\rfloor+1}^{k-1} D_{s_j}
		\end{equation*}
		Observe that if $j+i=\lfloor \gamma\rfloor$ for $\lfloor \gamma/2\rfloor+1\leqslant j\leqslant \lfloor \gamma\rfloor$, we have
		\begin{equation}\label{ij}
			D_{\frac{m+\gamma}{2}-j}=D_{\frac{m+\lfloor \gamma\rfloor-[\gamma]}{2}-i}.
		\end{equation}
		Using (\ref{ij}), the desired result follows by reindexing.
	\end{proof}
	
	Second, the operator $L_{2k}$ defined as a product of $L_{\gamma-2j}$ has a nice factorization (\ref{factor}) which implies that $L_{2k}$ is formally self-adjoint with respect to the weighted measure $\rho^{1-2[\gamma]}dzdtd\rho$ on $\Omega_{n+1}$. The following result can be deduced from Theorem 1.3 \cite{LY22} easily. In order to make this paper self-contained, we prove it in an alternative way.
	
	\begin{lemma}\label{factorization}
		Let $\gamma\in (0,m)\backslash \N$ and $L_{2k}$ be the weighted poly-sublaplacian defined in (\ref{weighted}). Then
		\begin{equation}\label{factor}
			L_{2k}=\prod_{j=0}^{k-1}\left(L_{[\gamma]}+i\left(k-1-2j\right)T\right),
		\end{equation}
		where $L_\mu$ is given in (\ref{L}) and $T=2\partial_t$.
	\end{lemma}

	\begin{proof}
		We prove the claim by induction. Direct calculation shows that (\ref{factor}) holds for $\gamma \in (0,2)$. Suppose that (\ref{factor}) holds for $\gamma\in (0,k-2)$ with $k\in\N$ and $k>2$. 
		
		For convenience, we denote the operator $\rho^{-1}\partial_\rho$ by $Y$. First observe the commutator identity
		\begin{equation}\label{Y}
			\left[Y, L_\mu\right]=2\left(Y^2+\partial_t^2\right), \quad \forall \mu\in \R. 
		\end{equation}
		A direct computation using (\ref{weighted}) yields
		\begin{equation}
			\begin{split}
				L_{2k}=&L_{\gamma-2k+2}\circ \prod_{j=1}^{k-2}\left(L_{\gamma-k+1}+i\left(k-1-2j\right)T\right)\circ L_\gamma\\
				=&\left(L_{[\gamma]}+2(k-1)Y \right)\circ L_{2(k-2)}\circ \left(L_{[\gamma]}-2(k-1)Y \right)\\
				=&L_{2(k-2)}\circ L_{[\gamma]}^2+ 2(k-1)\left[Y, L_{2(k-2)}\circ L_{[\gamma]} \right]\\
				&-4(k-1)^2 Y\circ L_{2(k-2)}\circ Y.
			\end{split}
		\end{equation}
		Hence, it suffices to show that for $k\in \N$ and $k>2$
		\begin{equation}\label{commutator}
			\left[Y, L_{2(k-2)}\circ L_{[\gamma]}\right]=2(k-1)Y\circ L_{2(k-2)}\circ Y+2(k-1) L_{2(k-2)}\circ \partial_t^2.
		\end{equation}
		By induction, suppose that (\ref{commutator}) holds for $k-2$, then
		\begin{equation}\label{cal1}
			\begin{split}
				\left[Y, L_{2k}\circ L_{[\gamma]}\right]&=\left[Y, \left(L_{[\gamma]}^2+(k-1)^2T^2\right)\circ L_{2(k-2)}\circ L_{[\gamma]}\right]\\
				&=\left[Y, L_{[\gamma]}^2+(k-1)^2T^2\right]\circ L_{2(k-2)}\circ L_{[\gamma]}\\
				&+\left(L_{[\gamma]}^2+(k-1)^2T^2\right)\circ \left[Y, L_{2(k-2)}\circ L_{[\gamma]}\right].
			\end{split}
		\end{equation}
		Notice that $T$ can commute with $Y$ and $L_\mu$ for all $\mu\in \R$. Hence
		\begin{equation}\label{cal2}
			\begin{split}
				\left[Y, L_{[\gamma]}^2+(k-1)^2T^2\right]&=\left[Y, L_{[\gamma]}\right]\circ L_{[\gamma]}+ L_{[\gamma]}\circ \left[Y, L_{[\gamma]}\right]\\
				&=2\left(Y^2\circ L_{[\gamma]}+L_{[\gamma]}\circ Y^2\right)+4L_{[\gamma]}\circ \partial_t^2.
			\end{split}
		\end{equation}
		Plugging into (\ref{cal2}) and (\ref{commutator}) into (\ref{cal1}) yields
		\begin{equation}
			\begin{split}
				\left[Y, L_{2k}\circ L_{[\gamma]}\right]&=2Y^2\circ L_{[\gamma]}^2\circ L_{2(k-2)}+2L_{[\gamma]}\circ Y^2\circ L_{2(k-2)}\circ L_{[\gamma]}\\
				&+2(k-1)\left(L_{[\gamma]}^2+(k-1)^2T^2\right)\circ Y\circ L_{2(k-2)}\circ Y\\
				&+4L_{[\gamma]}^2\circ L_{2(k-2)}\circ \partial_t^2
				+2(k-1) L_{2k}\circ \partial_t^2.
			\end{split}
		\end{equation}
		Besides, we have
		\begin{equation}
			Y^2\circ L_{[\gamma]}^2\circ L_{2(k-2)}=Y\circ L_{[\gamma]}^2\circ L_{2(k-2)}\circ Y+ Y\circ \left[Y,L_{[\gamma]}^2\circ L_{2(k-2)}\right],
		\end{equation}
		\begin{equation}
			\begin{split}
				L_{[\gamma]}\circ Y^2\circ &L_{2(k-2)}\circ L_{[\gamma]}=Y\circ L_{[\gamma]}^2\circ L_{2(k-2)}\circ Y\\
				&+Y\circ L_{[\gamma]}\circ \left[Y,L_{2(k-2)}\circ L_{[\gamma]}\right]\\
				&+ \left[L_{[\gamma]},Y\right]\circ Y\circ L_{[\gamma]}\circ L_{2(k-2)},
			\end{split}
		\end{equation}
		and 
		\begin{equation}
			\left(L_{[\gamma]}^2+(k-1)^2T^2\right)\circ Y\circ L_{2(k-2)}\circ Y=Y\circ L_{2k}\circ Y+ \left[L_{[\gamma]}^2,Y\right]\circ L_{2(k-2)}\circ Y.
		\end{equation}
		
		Combining them together yields
		\begin{equation}\label{cal3}
			\begin{split} 
				\left[Y, L_{2k}\circ L_{[\gamma]}\right]=&2(k+1)Y\circ L_{2k}\circ Y+\mathrm{I}+\mathrm{II}+\mathrm{III}\\
				&+4L_{[\gamma]}^2\circ L_{2(k-2)}\circ \partial_t^2
				+2(k-1) L_{2k}\circ \partial_t^2\\
				&-4(k-1)^2 Y\circ L_{2(k-2)}\circ Y \circ T^2,
			\end{split}
		\end{equation}
		where
		\begin{equation}
			\begin{split}
				\mathrm{I} =&2Y\circ \left[Y,L_{[\gamma]}^2\circ L_{2(k-2)}\right]\\
				\mathrm{II} =&2Y\circ L_{[\gamma]}\circ \left[Y,L_{2(k-2)}\circ L_{[\gamma]}\right] \\
				&+ 2\left[L_{[\gamma]},Y\right]\circ Y\circ L_{[\gamma]}\circ L_{2(k-2)}\\
				\mathrm{III}=&2(k-1)\left[L_{[\gamma]}^2,Y\right]\circ L_{2(k-2)}\circ Y.
			\end{split}
		\end{equation}
		Using (\ref{commutator}) and (\ref{Y}), we find that
		\begin{equation}
			\begin{split}
				\mathrm{I}+\mathrm{II}=&8(k-1)Y\circ L_{[\gamma]}\circ Y \circ L_{2(k-2)}\circ Y\\
				&+8(k-1)Y\circ L_{[\gamma]}\circ L_{2(k-2)}\circ \partial_t^2\\
				\mathrm{III}=&-4(k-1)Y^2\circ L_{[\gamma]}\circ L_{2(k-2)}\circ Y\\
				&-4(k-1)L_{[\gamma]}\circ Y^2\circ L_{2(k-2)}\circ Y\\
				&-8(k-1) L_{[\gamma]}\circ L_{2(k-2)}\circ Y\circ \partial_t^2
			\end{split}
		\end{equation}
		Notice that
		\begin{equation}
			\begin{split}
				Y\circ L_{[\gamma]}\circ Y \circ L_{2(k-2)}\circ Y=&Y^2\circ L_{[\gamma]}\circ L_{2(k-2)}\circ Y\\
				&-2Y\circ \left(Y^2+\partial_t^2\right)\circ L_{2(k-2)}\circ Y\\
				-L_{[\gamma]}\circ Y^2\circ L_{2(k-2)}\circ Y=&-Y^2\circ L_{[\gamma]}\circ L_{2(k-2)}\circ Y\\
				&+4Y\circ \left(Y^2+\partial_t^2\right)\circ L_{2(k-2)}\circ Y.
			\end{split}
		\end{equation}
		We conclude that
		\begin{equation}
			\begin{split}
				\mathrm{I}+\mathrm{II}+\mathrm{III}&=8(k-1)\left[Y, L_{2(k-2)}\circ L_{[\gamma]}\right]\circ \partial_t^2\\
				&=16(k-1)^2Y\circ L_{2(k-2)}\circ Y\circ \partial_t^2+16(k-1)^2 L_{2(k-2)}\circ \partial_t^2.
			\end{split}
		\end{equation}
		Observe that
		\begin{equation}
			\begin{split}
				&4L_{[\gamma]}^2\circ L_{2(k-2)}\circ \partial_t^2
				+2(k-1) L_{2k}\circ \partial_t^2\\
				&=2(k+1) L_{2k}\circ \partial_t^2-16(k-1)^2L_{2(k-2)}\circ \partial_t^2.
			\end{split}
		\end{equation}
		Finally, (\ref{commutator}) holds for $k$.
	\end{proof}
	
	We now prove that elements of the kernel $D_{s-j}$ are also in the kernel of $L_{2k}$ when weighted against a suitable power of $\rho$.
	\begin{lemma}\label{kernel}
		Let $\gamma\in (0,m)\backslash \N$ and set $k:=\lfloor \gamma\rfloor+1$. Denote
		\begin{equation*}
			\mathcal{I}^\gamma:=\left\{\gamma-2j|j\in[0,\lfloor \gamma/2\rfloor]\right\}\cup \left\{\lfloor \gamma\rfloor-[\gamma]-2j|j\in[0, \lfloor \gamma\rfloor-\lfloor \gamma/2\rfloor-1]\right\}.
		\end{equation*}
		For each $\tilde{\gamma}\in \mathcal{I}^\gamma$ it holds that
		\begin{equation*}
			L_{2k} \circ \rho^{-m+\gamma}\circ \mathcal{P}(\frac{m+\tilde{\gamma}}{2})=0.
		\end{equation*}
	\end{lemma}
	\begin{proof}
		It is clear from Lemma \ref{kernel} that
		\begin{equation*}
			L_{2k}^+\circ \mathcal{P}(\frac{m+\tilde{\gamma}}{2})=0,
		\end{equation*}
		for each $\tilde{\gamma}\in \mathcal{I}^\gamma$. The desired result follows from a repeated application of the transformation law (\ref{transformation}):
		\begin{equation}
			4^{-k} L_{2k} \circ \rho^{-m+\gamma}=\rho^{-m+\gamma-2k} \circ L^+_{2k}.
		\end{equation}
	\end{proof}
	
	\section{Boundary operators on the Siegel domain}
	In this section we introduce the boundary operators associated to the weighted ploy-sublaplacian $L_{2k}$ determined by $\gamma\in (0,2)\backslash \N$. By Lemma \ref{kernel}, the kernel of $L_{2k}$ contains solutions of the Poisson equation (\ref{expansion}) for any $\tilde{\gamma}\in \mathcal{I}^\gamma$. Our boundary operators are designed to pick out the functions $F(\cdot, 0)$ and $G(\cdot, 0)$ of solutions to (\ref{expansion}). They also give rise to formally self-adjoint boundary value problems. To that end, it is convenient to introduce the space
	\begin{equation*}
		\mathcal{C}^{2\gamma}=C^{\infty}_{\mathrm{even}}(\overline{\Omega}_{n+1})+\rho^{2[\gamma]}C^{\infty}_{\mathrm{even}}(\overline{\Omega}_{n+1})
	\end{equation*}
	associated to a given $\gamma\in (0,m)\backslash \N$, where $C^{\infty}_{\mathrm{even}}$ denotes the space of smooth functions on $\overline{\Omega}_{n+1}$ whose Taylor series expansions in $\rho$ at $\rho=0$ contain only even terms. Note that\\
	(1) if $\gamma\in \frac{1}{2}+\N_0$, then $\mathcal{C}^{2\gamma}=C^{\infty}(\overline{\Omega}_{n+1})$; and\\
	(2) for any $\tilde{\gamma}\in \mathcal{I}^\gamma$, it holds that $\mathcal{P}\left(\frac{m+\tilde{\gamma}}{2}\right):C^{\infty}(\H^n)\cap S^{\tilde{\gamma},2}(\H^n)\to \mathcal{C}^{2\gamma}$.
	
	The second point means that the space $\mathcal{C}^{2\gamma}$ is well-suited to studying all of the scattering problems formed from the factors of the complex hyperbolic poly-Poisson operator $L_{2k}^+$. Let $\iota^*:\mathcal{C}^{2\gamma}\to C^{\infty}(\H^n)$ denote the restriction operator, $\iota^*U=U|_{\H^n}$. Our boundary operators are elements of the set 
	\begin{equation*}
		\begin{split}
			\mathcal{B}^{2\gamma}:=&\left\{B_{2j}^{2\gamma}:\mathcal{C}^{2\gamma}\to C^{\infty}(\H^n)|j\in[0, \lfloor \gamma\rfloor]\right\}\\
			& \cup \left\{B_{2[\gamma]+2j}^{2\gamma}:\mathcal{C}^{2\gamma}\to C^{\infty}(\H^n)|j\in[0, \lfloor \gamma\rfloor]\right\}
		\end{split}
	\end{equation*}
	defined as follows:
	\begin{definition}\label{boundary}
		Set 
		\begin{equation*}
			\mathcal{R}:=\partial^2_\rho+(1-2[\gamma])\partial_\rho, \quad \quad \mathcal{T}:=\mathcal{R}+\rho^2\partial^2_t.
		\end{equation*}
For $\gamma\in (0,1)$, we define associated boundary operators by
		\begin{equation}
               B_{0}^{2\gamma}=\iota^*, \quad
				B_{2\gamma}^{2\gamma}=-\iota^* \circ \rho^{1-2[\gamma]}\partial_{\rho}.
		\end{equation}
For $\gamma\in (1,2)$, we define associated boundary operators by
		\begin{equation}
               \begin{split}
               B_{0}^{2\gamma}=\iota^*, \quad & B_{2}^{2\gamma}=- \iota^* \circ \mathcal{T}+\frac{1-[\gamma]}{[\gamma]}\Delta_b\circ B_{0}^{2\gamma}, \\ B_{2[\gamma]}^{2\gamma}=-\iota^* \circ \rho^{1-2[\gamma]}\partial_{\rho}, \quad  &B_{2\gamma}^{2\gamma}= \iota^* \circ \rho^{1-2[\gamma]}\partial_{\rho} \mathcal{T}-\frac{1+[\gamma]}{[\gamma]}\Delta_b\circ B_{2[\gamma]}^{2\gamma}.
               \end{split}
		\end{equation}
	\end{definition}
	
	The first goal of this section is to show that the boundary operators $\mathcal{B}^{2\gamma}$ are relevant for picking out the Dirichlet data $F(\cdot,0)$ and the Neumann data $G(\cdot,0)$ of solutions of the Poisson equation $\mathcal{P}\left(\frac{m+\tilde{\gamma}}{2}\right)$ for any $\tilde{\gamma}\in \mathcal{I}^\gamma$. This is accomplished by the following two propositions.
	\begin{proposition}\label{energy1}
		Let $\gamma\in (0,2)\backslash \N$. Let $V=\rho^{-m+\gamma}\mathcal{P}\left(\frac{m+{\gamma}}{2}\right)(f)$ for some $f\in C^{\infty}(\H^n)\cap S^{\gamma,2}(\H^n)$. It holds that
		\begin{equation}\label{2j}
			B_{0}^{2\gamma}(V)=2^{\frac{\gamma-m}{2}}f,
		\end{equation}
		\begin{equation}\label{2j2}
			B_{2\gamma}^{2\gamma}(V)=(-1)^{\lfloor \gamma\rfloor+1}2^{-\frac{m+\gamma}{2}}2^{2\lfloor \gamma\rfloor+1}\lfloor \gamma\rfloor!\frac{\Gamma(\gamma+1)}{\Gamma([\gamma])}\hat{f}
		\end{equation}
		where $\hat{f}=\mathcal{S}\left(s\right)f$ with $s=\frac{m+{\gamma}}{2}$. Moreover, $B_{2\alpha}^{2\gamma}(V)=0$ for all $B_{2\alpha}^{2\gamma}\in \mathcal{B}^{2\gamma}\backslash \{B_{0}^{2\gamma}, B_{2\gamma}^{2\gamma}\}$.
	\end{proposition}
	\begin{proof}
		To begin, note that
		\begin{equation}\label{T}
			\mathcal{R}(\rho^{2j})=4j(j-[\gamma])\rho^{2j-2}, \quad \quad \mathcal{R}(\rho^{2[\gamma]+2j})=4j(j+[\gamma])\rho^{2[\gamma]+2j-2}.
		\end{equation}

		By (\ref{expansion}), it holds that
		\begin{equation}\label{V}
			V=\rho^{-m+\gamma}\mathcal{P}\left(\frac{m+{\gamma}}{2}\right)f=\left(\frac{1}{2}\right)^{\frac{m-\gamma}{2}}F+ \left(\frac{1}{2}\right)^{\frac{m+\gamma}{2}}\rho^{2\gamma}G,
		\end{equation}
		where by Proposition \ref{formalep},
		\begin{equation*}
			F=\sum_{l=0}^{\infty}(-1)^l\frac{\Gamma(1-\gamma)}{2^l l!\Gamma(l+1-\gamma)}\rho^{2l} P^{s}_l(f),
		\end{equation*}
		and 
		\begin{equation*}
			G=\sum_{l=0}^{\infty}(-1)^l\frac{\Gamma(1+\gamma)}{2^ll!\Gamma(l+1+\gamma)}\rho^{2l}G^{s}_l(\hat{f}).
		\end{equation*}
		We separate the proof into two cases:
		
		First consider $B_{2l}^{2\gamma}$ for $l\in[0, \lfloor \gamma\rfloor]$. By the definition of $B_{0}^{2\gamma}$, we know that (\ref{2j}) holds. When $\gamma\in (1,2)$ and $l=1$, using (\ref{2j}) and (\ref{definitionP}),we compute that
		\begin{equation}
			B_{2}^{2\gamma}(V)=-2^{\frac{\gamma-m}{2}}2\frac{\Gamma(2-[\gamma])}{\Gamma(1-[\gamma])}f_1-\frac{\Gamma(2-[\gamma])\Gamma(1-\gamma)}{\Gamma(1-[\gamma])\Gamma(2-\gamma)}\Delta_b \circ B_{0}^{2\gamma}(V)=0.
		\end{equation}
		
		Next consider $B_{2[\gamma]+2l}^{2\gamma}$ for $l\in[0, \lfloor \gamma\rfloor]$. Notice that from (\ref{T}) and (\ref{V}) we immediately deduce that 
		\begin{equation}
			\begin{split}
				B_{2\gamma}^{2\gamma}(V)=-2^{-\frac{m+\gamma}{2}}2\gamma \hat{f}, \quad \gamma\in (0,1);
			\end{split}
		\end{equation}
		i.e., (\ref{2j2}) holds. When $\gamma\in (1,2)$, by the definition of $B_{2[\gamma]+2l}^{2\gamma}$, we immediately have $B_{2[\gamma]}^{2\gamma}(V)=0$. Besides, we compute that
		\begin{equation}\label{cal5}
			B_{2\gamma}^{2\gamma}(V)=2^{-\frac{m+\gamma}{2}}2^3\frac{\Gamma([\gamma]+2)}{\Gamma([\gamma])}g_{0};
		\end{equation}
		i.e., (\ref{2j2}) holds.
	\end{proof}
	
	\begin{proposition}\label{energy2}
		Let $\gamma\in (1,2)$. Let $\tilde{V}=\rho^{-m+\gamma}\mathcal{P}\left(\frac{m+{\tilde{\gamma}}}{2}\right)\tilde{f}$ for some $\tilde{f}\in C^{\infty}(\H^n)\cap S^{\tilde{\gamma},2}(\H^n)$. It holds that
		\begin{equation}
			B_{2[\gamma]}^{2\gamma}(\tilde{V})=-2^{\frac{\tilde{\gamma}-m}{2}}2[\gamma]\tilde{f}
		\end{equation}
		\begin{equation}
			B_{2}^{2\gamma}(\tilde{V})=-2^{-\frac{m+\tilde{\gamma}}{2}}4(1-[\gamma])\hat{\tilde{f}},
		\end{equation}
		where $\hat{\tilde{f}}=\mathcal{S}(\tilde{s})\tilde{f}$ with $\tilde{s}=\frac{m+\tilde{\gamma}}{2}$, $\tilde{\gamma}=1-[\gamma]$.  Moreover, $B_{2\alpha}^{2\gamma}(V)=0$ for all $B_{2\alpha}^{2\gamma}\in \mathcal{B}^{2\gamma}\backslash \{B_{2[\gamma]}^{2\gamma}, B_{2}^{2\gamma}\}$.
	\end{proposition}
	\begin{proof}
		It follows from (\ref{expansion}) that
		\begin{equation}\label{V2}
			\tilde{V}=\rho^{-m+\gamma}\mathcal{P}\left(\frac{m+{\tilde{\gamma}}}{2}\right)f=\left(\frac{1}{2}\right)^{\frac{m-\tilde{\gamma}}{2}}\rho^{2[\gamma]}\tilde{F}+ \left(\frac{1}{2}\right)^{\frac{m+\tilde{\gamma}}{2}}\rho^{2}\tilde{G}
		\end{equation}
		where by Proposition \ref{formalep},
		\begin{equation}
			\tilde{F}=\sum_{l=0}^{\infty}(-1)^l\frac{\Gamma(1-\tilde{\gamma})}{2^l l!\Gamma(l+1-\tilde{\gamma})}\rho^{2l} P^{\tilde{s}}_l(\tilde{f}),
		\end{equation}
		and 
		\begin{equation*}
			\tilde{G}=\sum_{l=0}^{\infty}(-1)^l\frac{\Gamma(1+\tilde{\gamma})}{2^ll!\Gamma(l+1+\tilde{\gamma})}\rho^{2l}G^{\tilde{s}}_l(\hat{\tilde{f}}).
		\end{equation*}
		Using (\ref{T}), (\ref{V2}) and Definition \ref{boundary}, direct computation yields the desired result.
	\end{proof}
\begin{remark}\label{obs}
    Our boundary operators are designed to pick out the functions $F(\cdot,0)$ and $G(\cdot,0)$ of solutions of (\ref{expansion}). To guarantee that $B_{2l}^{2\gamma}(V)=0$ for $l>j$ (see \cite[Propositon 3.2]{Case20} for more details), we need to make the cancellation by composing $B_{2j}^{2\gamma}$ with $P^s_{l-j}=(-\overline{\Delta})^{l-j}$ because the output of $\iota^*\circ T^l$ is exactly $f_{l-j}$. However, in the CR setting, the presence of the term $\rho^2\partial^2_t$ affects the output of leading terms $\iota^*\circ \mathcal{T}^l$ and $\iota^* \circ \rho^{1-2[\gamma]}\partial_{\rho} \mathcal{T}^l$, $l\geqslant 2$, which implies that the operator to be composed is no longer $P^s_{l-j}$.
\end{remark}	
	
		Let 
		\begin{equation}\label{Energy}
			\begin{split}
				\mathcal{Q}_{2\gamma}(U,V):=&-\int_{\Omega_{n+1}}UL_{2}V \rho^{1-2\gamma}dzdtd\rho + \oint_{\H^n} B_{0}^{2\gamma}(U) B_{2\gamma}^{2\gamma}(V)dzdt,\\
				\mathcal{Q}_{2\gamma}(U,V):=&\int_{\Omega_{n+1}}UL_{4}V \rho^{1-2[\gamma]}dzdtd\rho \\
				&+ \oint_{\H^n} B_{0}^{2\gamma}(U) B_{2\gamma}^{2\gamma}(V)dzdt-\oint_{\H^n} B_{2[\gamma]}^{2\gamma}(U) B_{2}^{2\gamma}(V)dzdt
			\end{split}
		\end{equation}
		be the associated Dirichlet energy determined by $\gamma\in (0,1)$ and $\gamma\in (1,2)$, respectively. The second goal of this section is to prove that $\mathcal{Q}_{2\gamma}$ is symmetric. This implies that the boundary value problem involving $L_{2k}$ and $\mathcal{B}^{2\gamma}$ are variational. The proof that $\mathcal{Q}_{2\gamma}$ is symmetric is essentially integration by parts on the Siegel domain $\Omega_{n+1}$ with flat metric. To that end, it is useful to express $\iota^*\circ L_{[\gamma]}$ and $\iota^* \circ \rho^{1-2[\gamma]}\partial_{\rho} L_{[\gamma]}$ in terms of boundary operators. 
		\begin{proposition}\label{4.4}
			For $\gamma\in (1,2)$, it holds that
			\begin{equation}\label{Lj}
				\begin{split}
				\iota^* =B_{0}^{2\gamma},\quad	&\iota^*\circ L_{[\gamma]}=-B_{2}^{2\gamma}+\frac{1}{[\gamma]}\Delta_b\circ B_{0}^{2\gamma}, \\ 
                  \iota^* \circ \rho^{1-2[\gamma]}\partial_{\rho} =-B_{2[\gamma]}^{2\gamma},\quad &\iota^* \circ \rho^{1-2[\gamma]}\partial_{\rho} L_{[\gamma]}= B_{2\gamma}^{2\gamma}+\frac{1}{[\gamma]}\Delta_b\circ B_{2[\gamma]}^{2\gamma}.
				\end{split}
			\end{equation}
		\end{proposition}
		\begin{proof}
			Note that
			\begin{equation}\label{expL}
				\iota^*\circ L_{[\gamma]}=\iota^*\circ \left(\mathcal{T}+\Delta_b\right)= \iota^*\circ\mathcal{T} + \iota^* \circ \Delta_b=\iota^*\circ\mathcal{T} + \Delta_b\circ\iota^* .
			\end{equation}
            Then, (\ref{Lj}) follows from Definition \ref{boundary} directly. 
		\end{proof}
		
		We now prove that $\mathcal{Q}_{2\gamma}$ is symmetric by giving an explicit formula for $\mathcal{Q}_{2\gamma}$. Especially notable in this formula is that the boundary integration involves only the Dirichlet data
		\begin{equation*}
			\mathcal{B}^{2\gamma}_D:=\left\{B^{2\gamma}_{2\alpha}\in \mathcal{B}^{2\gamma}| 0\leqslant \alpha <\gamma/2\right\}
		\end{equation*}
		of the inputs.
		
		\begin{theorem}\label{symmetry}
			Given $\gamma\in (0,1)$, we have 
			\begin{equation}\label{e1}
				\begin{split}
					&\mathcal{Q}_{2\gamma}(U,V)\\
					&=\int_{\Omega_{n+1}}\left(\partial_{\rho}U \partial_{\rho}V+\frac{1}{2}\sum_{j=1}^nX_j(U)X_j(V)+Y_j(U)Y_j(V)+\rho^2\partial_tU\partial_t V \right)\rho^{1-2\gamma}dzdtd\rho.
				\end{split}
			\end{equation}
			Likewise, given $\gamma\in (1,2)$, we have
			\begin{equation}\label{e2}
				\begin{split}
					\mathcal{Q}_{2\gamma}(U,V)&=\int_{\Omega_{n+1}} \left(L_{[\gamma]}(U)L_{[\gamma]}(V)-T(U) T(V)\right)\rho^{1-2[\gamma]}dzdtd\rho\\
					&-\frac{1}{[\gamma]}\oint_{\H^n}\Delta_b \left(B_{0}^{2\gamma}(U) B_{2[\gamma]}^{2\gamma}(V)+B_{2[\gamma]}^{2\gamma}(U) B_{0}^{2\gamma}(V)\right)dzdt
				\end{split}
			\end{equation}
		\end{theorem}
		\begin{proof}
			Given $\gamma\in (0,1)$, using the integration by parts on the Siegel domain $\Omega_{n+1}$ with flat metric, we have
			\begin{equation}\label{int1}
				\begin{split}
					&-\int_{\Omega_{n+1}} U \left(\partial^2_\rho+(1-2\gamma)\rho^{-1}\partial_{\rho}\right)V\rho^{1-2\gamma}dzdtd\rho\\
					&=\int_{\Omega_{n+1}}\partial_{\rho}U \partial_{\rho}V\rho^{1-2\gamma}dzdtd\rho-\oint_{\H^n} B_{0}^{2\gamma}(U) B_{2\gamma}^{2\gamma}(V)dzdt.
				\end{split}
			\end{equation}
			\begin{equation}\label{int2}
				-\int_{\Omega_{n+1}} U \Delta_b V\rho^{1-2\gamma}dzdtd\rho
				=\frac{1}{2}\int_{\Omega_{n+1}}\sum_{j=1}^nX_j(U)X_j(V)+Y_j(U)Y_j(V) \rho^{1-2\gamma}dzdtd\rho.
			\end{equation}
			\begin{equation}\label{int3}
				-\int_{\Omega_{n+1}} U \rho^2\partial^2_t V \rho^{1-2\gamma}dzdtd\rho=\int_{\Omega_{n+1}} \rho^2\partial_tU\partial_t V \rho^{1-2\gamma}dzdtd\rho.
			\end{equation}
			Combining (\ref{int1}), (\ref{int2}) and (\ref{int3}) yields (\ref{e1}).
			
			Similarly, given $\gamma\in (1,2)$, by Lemma \ref{factorization}, we know that $L_4=L^2_{[\gamma]}+T^2$. Notice that $L_{[\gamma]}$ is formally self-adjoint with respect to the weighted measure $\rho^{1-2[\gamma]}dzdtd\rho$. By the integration by parts, we have that
			\begin{equation}\label{int4}
				\begin{split}
					&\int_{\Omega_{n+1}} U L^2_{[\gamma]}(V)\rho^{1-2[\gamma]}dzdtd\rho=\int_{\Omega_{n+1}} L_{[\gamma]}(U)L_{[\gamma]}(V)\rho^{1-2[\gamma]}dzdtd\rho\\
					&+\oint_{\H^n}B_{0}^{2\gamma}(U) B_{2[\gamma]}^{2\gamma}(L_{[\gamma]}V)dzdt-\oint_{\H^n}B_{2[\gamma]}^{2\gamma}(U) B_{0}^{2\gamma}(L_{[\gamma]}V)dzdt,\\
					&\int_{\Omega_{n+1}} U T^2(V)\rho^{1-2[\gamma]}dzdtd\rho=-\int_{\Omega_{n+1}} T(U) T(V)\rho^{1-2[\gamma]}dzdtd\rho.
				\end{split}
			\end{equation}
			By Propositon \ref{4.4}, we know that
			\begin{equation}\label{4.24}
				\begin{split}
					B_{0}^{2\gamma}\circ L_{[\gamma]}&=-B_{2}^{2\gamma}+\frac{1}{[\gamma]}\Delta_b\circ B_{0}^{2\gamma},\\
					B_{2[\gamma]}^{2\gamma}\circ L_{[\gamma]}&=-B_{2\gamma}^{2\gamma}-\frac{1}{[\gamma]}\Delta_b\circ B_{2[\gamma]}^{2\gamma}.    
				\end{split} 
			\end{equation}
			Plugging (\ref{4.24}) into (\ref{int4}) yields (\ref{e2}).
		\end{proof}
		
			
			\section{The generalized extension problem}
			The main result of this section is that solutions of the Dirichlet boundary value problems
			\begin{equation}\label{Dirichlet1}
				\left\{
				\begin{array}{ll}
					L_{2}V=0, &  ~~\mbox{on}~~ \Omega_{n+1}, \quad \gamma\in (0,1),\\
					B_{0}^{2\gamma}(V)=\phi^{}, & ~~\mbox{on}~~ \H^{n},
				\end{array}
				\right.
			\end{equation}
			and
			\begin{equation}\label{Dirichlet2}
				\left\{
				\begin{array}{ll}
					L_{4}V=0, &  ~~\mbox{on}~~ \Omega_{n+1}, \quad \gamma\in (1,2),\\
					B_{0}^{2\gamma}(V)=\phi, & ~~\mbox{on}~~ \H^{n},\\
					B_{2[\gamma]}^{2\gamma}(V)=\psi, &  ~~\mbox{on}~~ \H^{n},
				\end{array}
				\right.
			\end{equation}
			can be used to recover the CR fractional GJMS operators. To that end, we first characterize the solutions of (\ref{Dirichlet1}) and (\ref{Dirichlet2}).
			\begin{theorem}\label{solution1}
				Let $\gamma\in (0,1)$. Given functions $\phi\in C^{\infty}(\H^n)\cap S^{\gamma,2}(\H^n)$, there is a unique solution $V$ of (\ref{Dirichlet1}). Indeed, 
				\begin{equation}
					V=2^{-\frac{\gamma-m}{2}}\rho^{-m+\gamma}\mathcal{P}\left(\frac{m+{\gamma}}{2}\right)\phi.
				\end{equation}
			\end{theorem}
			\begin{proof}
				It follows from Lemma \ref{kernel} and Proposition \ref{energy1} that $V$ satisfies (\ref{Dirichlet1}). The uniqueness follows from the fact that $\mathcal{Q}_{2\gamma}(W,W)\geqslant 0$ with equality holds if and only if $W=0$.
			\end{proof}
			
			\begin{theorem}\label{solution2}
				Let $\gamma\in (1,2)$. Given $\phi\in C^{\infty}(\H^n)\cap S^{\gamma,2}(\H^n)$ and $\psi\in C^{\infty}(\H^n)\cap S^{1-[\gamma],2}(\H^n)$, there is a unique solution $V$ of (\ref{Dirichlet2}). Indeed, 
				\begin{equation}
					V=2^{-\frac{\gamma-m}{2}}\rho^{-m+\gamma}\mathcal{P}\left(\frac{m+{\gamma}}{2}\right)\phi-2^{-\frac{\tilde{\gamma}-m}{2}}\frac{1}{2[\gamma]}\rho^{-m+\gamma}\mathcal{P}\left(\frac{m+1-[\gamma]}{2}\right)\psi,
				\end{equation}
    where $\tilde{\gamma}=1-[\gamma]$.
			\end{theorem}
			\begin{proof}
				It follows from Lemma \ref{kernel} and Proposition \ref{energy2} that $V$ satisfies (\ref{Dirichlet2}). Suppose now that $U$ is a solution of (\ref{Dirichlet2}). Then $W:=U-V$ solves
				\begin{equation}
					\left\{
					\begin{array}{ll}
						L_{4}W=0, &  ~~\mbox{on}~~ \Omega_{n+1}, \quad \gamma\in (1,2),\\
						B_{0}^{2\gamma}(W)=0, & ~~\mbox{on}~~ \H^{n},\\
						B_{2[\gamma]}^{2\gamma}(W)=0, &  ~~\mbox{on}~~ \H^{n}.
					\end{array}
					\right.
				\end{equation}
				From Theorem \ref{symmetry} and \cite[Theorem 1.5]{LY22}, we know that
				\begin{equation*}
					\mathcal{Q}_{2\gamma}(W,W)=\int_{\Omega_{n+1}} \left(L_{[\gamma]}(W)L_{[\gamma]}(W)-T(W) T(W)\right)\rho^{1-2[\gamma]}dzdtd\rho\geqslant 0,
				\end{equation*}
				with equality holds if and only if $W=0$. Therefore, $U=V$.
			\end{proof}
			We now present our generalized extension problem. In fact, the following result implies that the CR fractional GJMS operators can be determined without fully specifying the Dirichlet data.
			
			\begin{theorem}\label{Caffarelli-Silvestre-extension}
				\begin{enumerate}
					\item Given $\gamma\in (0,1)$, suppose that $V$ is a solution of (\ref{Dirichlet1}). It holds that
					\begin{equation*}
						B_{2\gamma}^{2\gamma}(V)=2^{1-2\gamma}\frac{\Gamma(1-\gamma)}{\Gamma(\gamma)}P^{\theta}_{\gamma}\phi.
					\end{equation*}
					\item Given $\gamma\in (1,2)$, $\tilde{\gamma}=1-[\gamma]$, suppose that $V$ is a solution of (\ref{Dirichlet2}). It holds that
					\begin{equation*}
						\begin{split}
							B_{2}^{2\gamma}(V)&=2^{1-2\tilde{\gamma}}\frac{\tilde{\gamma}}{1-\tilde{\gamma}}\frac{\Gamma(-\tilde{\gamma})}{\Gamma(\tilde{\gamma})}P^{\theta}_{\tilde{\gamma}}\psi,\\
							B_{2\gamma}^{2\gamma}(V)&=2^{3-2\gamma}\frac{\Gamma(2-\gamma)}{\Gamma(\gamma)}P^{\theta}_{\gamma}\phi.
						\end{split}
					\end{equation*}
				\end{enumerate}
			\end{theorem}
			\begin{proof}
				It follows from (\ref{GJMS}), Proposition \ref{energy1}, Proposition \ref{energy2}, Theorem \ref{solution1} and Theorem \ref{solution2} directly.
			\end{proof}
			
			\section{The sharp CR trace Sobolev inequalities}
			The purpose of this section is to use the boundary operators to prove CR sharp Sobolev trace inequalities. A key tool in this endeavor is the Dirichlet energy
			\begin{equation*}
				\mathcal{E}_{2\gamma}(U):=\mathcal{Q}_{2\gamma}(U,U),
			\end{equation*}
		where $\mathcal{Q}_{2\gamma}$ is givey by ~\eqref{Energy}. 
			Our first result is a Dirichlet principle for solutions of ~\eqref{Dirichlet1} and ~\eqref{Dirichlet2}.
	
	\begin{theorem}\label{Dirichlet-Principle}
If $\gamma\in (0,1)$, fix a function $\phi \in C^{\infty}\left(\H^{n}\right) \cap S^{\gamma,2}(\H^n) $ and denote
$$
\begin{gathered}
	\mathcal{C}_{D}^{2 \gamma}:=\left\{U \in \mathcal{C}^{2 \gamma} \mid B_{0}^{2 \gamma}(U)=\phi \right\} 
\end{gathered}
$$
Moreover, if $\gamma\in (1,2)$, then fix functions $\phi \in C^{\infty}\left(\H^{n}\right) \cap S^{\gamma,2}(\H^n)$, and $\psi \in C^{\infty}\left(\H^{n}\right) \cap S^{1-[\gamma],2}(\H^n)$, and denote

$$
\begin{gathered}
	\mathcal{C}_{D}^{2 \gamma}:=\left\{U \in \mathcal{C}^{2 \gamma} \mid B_{0}^{2 \gamma}(U)=\phi, \text { and } B_{2[\gamma]}^{2 \gamma}(U)=\psi \right\} .
\end{gathered}
$$

Then it holds that
$$
\mathcal{E}_{2 \gamma}(U) \geq \mathcal{E}_{2 \gamma}\left(U_{D}\right)
$$
for all $U \in \mathcal{C}_{D}^{2 \gamma}$, where $U_{D} \in \mathcal{C}_{D}^{2 \gamma}$ is the unique solution of ~\eqref{Dirichlet1} and ~\eqref{Dirichlet2}.
	\end{theorem}

\begin{proof}
	For  $\gamma\in (0,1)$, fix an element $U_{0} \in \mathcal{C}_{D}^{2 \gamma}$ and set

$$
\begin{aligned}
	\mathcal{C}_{0}^{2 \gamma}=\left\{U \in \mathcal{C}^{2 \gamma} \mid B_{0}^{2 \gamma}(U)=0,  \text { and } B_{2[\gamma]}^{2 \gamma}(U)=0 \right\} .
\end{aligned}
$$
Observe that
$$
\mathcal{C}_{D}^{2 \gamma}=U_{0}+\mathcal{C}_{0}^{2 \gamma}
$$
Let $V \in \mathcal{C}_{0}^{2 \gamma}$. It follows from ~\eqref{Energy} and Theorem~\ref{symmetry} that
\begin{equation*}
    \begin{split}
	\mathcal{E}_{2\gamma}(V):=&-\int_{\Omega_{n+1}}UL_{2}V \rho^{1-2\gamma}dzdtd\rho, \quad \gamma\in (0,1),\\
	\mathcal{E}_{2\gamma}(V):=&\int_{\Omega_{n+1}}UL_{4}V \rho^{1-2[\gamma]}dzdtd\rho, \quad \gamma \in (1,2),
			\end{split}
		\end{equation*}
and that
$$
\frac{d}{d t} \mathcal{E}_{2 \gamma}(U+t V)=2 \mathcal{E}_{2 \gamma}(V) \geq 0.
$$
Moreover, by \cite[Theorem 1.5]{LY22}, equality holds if and only if $V \equiv 0$, and hence $\mathcal{E}_{2 \gamma}$ is strictly convex in $\mathcal{C}_{D}^{2 \gamma}$. Since the solution $U_{D} \in \mathcal{C}_{D}^{2 \gamma}$ of ~\eqref{Dirichlet1} and ~\eqref{Dirichlet2} is a critical point of $\mathcal{E}_{2 \gamma}: \mathcal{C}_{D}^{2 \gamma} \rightarrow \mathbb{R}$, the result follows.	
\end{proof}
	
The following corollary, obtained by evaluating $\mathcal{E}_{2 \gamma}(U)$ using Theorem \ref{Caffarelli-Silvestre-extension} gives a sharp Sobolev trace inequality for the corresponding embedding.
	\begin{corollary}\label{cover}
		Let $\gamma \in(0,1)$. Then
		\begin{equation}\label{sharp1}
			\mathcal{E}_{2 \gamma}(U) \geqslant  2^{1-2\gamma}\frac{\Gamma(1-\gamma)}{\Gamma(\gamma)}\oint_{\H^{n}} \phi P^{\theta}_{\gamma}\phi dzdt,
		\end{equation}
		for all $U \in \mathcal{C}_{D}^{2 \gamma}$, where $\phi=B_{0}^{2\gamma}(U)$. Moreover, equality holds if and only if $U$ is the solution of ~\eqref{Dirichlet1}.
	\end{corollary}

\begin{proof}
	Set $\mathcal{C}_{D}^{2 \gamma}=U+\mathcal{C}_{0}^{2 \gamma}$. By Theorem \ref{Dirichlet-Principle} there is a unique minimizer $U_{D}$ of $\mathcal{E}_{2 \gamma}: \mathcal{C}_{D}^{2 \gamma} \rightarrow \mathbb{R}$. Since $U_{D}$ satisfies ~\eqref{Dirichlet1}, we deduce from Theorem \ref{Caffarelli-Silvestre-extension} that
$$
\begin{aligned}
	 \mathcal{E}_{2 \gamma}\left(U_{D}\right)= 2^{1-2\gamma}\frac{\Gamma(1-\gamma)}{\Gamma(\gamma)}\oint_{\H^{n}} \phi P^{\theta}_{\gamma}\phi dzdt.
\end{aligned}
$$
The conclusion readily follows.	
\end{proof}

Note that Corollary \ref{cover} recovers the main result in \cite{FGMT15}. Similarly, we have the following corollary. 
	\begin{corollary}\label{cover2}
	Let $\gamma \in(1,2)$, $\tilde{\gamma}=1-[\gamma]$. Then
	$$
	\begin{aligned}
		& \mathcal{E}_{2 \gamma}(U) \geq  2^{1-2\tilde{\gamma}}\frac{\tilde{\gamma}}{1-\tilde{\gamma}}\frac{\Gamma(-\tilde{\gamma})}{\Gamma(\tilde{\gamma})}\oint_{\H^{n}} \psi P^{\theta}_{\tilde{\gamma}}\psi dzdt+2^{3-2\gamma}\frac{\Gamma(2-\gamma)}{\Gamma(\gamma)}\oint_{\H^{n}} \phi P^{\theta}_{\gamma}\phi dzdt,
	\end{aligned}
	$$
	for all $U \in \mathcal{C}_{D}^{2 \gamma}$, where 				\begin{equation*}
			\psi=B_{2[\gamma]}^{2\gamma}(U),\quad \phi=B_{0}^{2\gamma}(U).
	\end{equation*}
Moreover, equality holds if and only if $U$ is the solution of ~\eqref{Dirichlet2}. 
\end{corollary}
 
Combining Corollary \ref{cover} and Corollary \ref{cover2} with the sharp Sobolev inequalities on $\H^n$ \cite{FL10} we obtain the following inequality
\begin{corollary}\label{6.4}
We denote $C_{n,2\gamma}$ the sharp constant in CR Sobolev inequalities 
\begin{equation}
    \left(\int_{\H^n}|f|^pdzdt\right)^{\frac{2}{p}}\leqslant C_{n,2\gamma}\int_{\H^n}\bar{f} P^{\theta}_{\gamma} (f) dzdt, \quad p=\frac{2m}{m-\gamma}.
\end{equation}
    \begin{enumerate}
        \item Let $\gamma \in(0,1)$. Then
		\begin{equation}\label{sharp1}
			 \mathcal{E}_{2 \gamma}(U) \geq  2^{1-2\gamma}\frac{\Gamma(1-\gamma)}{\Gamma(\gamma)}C^{-1}_{n,2\gamma}\left(\oint_{\H^{n}} |\phi|^{\frac{2m}{m-\gamma}}dzdt\right)^{\frac{m-\gamma}{m}},
		\end{equation}
		for all $U \in \mathcal{C}_{D}^{2 \gamma}$, where $\phi=B_{0}^{2\gamma}(U)$. Moreover, equality holds if and only if $U$ is the solution of ~\eqref{Dirichlet1} and 
  \begin{equation*}
      \phi(z,t)=C\left(\frac{1}{(1+|z|^2)^2+t^2}\right)^{\frac{m-\gamma}{2}}
  \end{equation*}
  for some $C\in \C$ and up to CR automorphisms on $\H^n$.
  \item Let $\gamma \in(1,2)$, $\tilde{\gamma}=1-[\gamma]$. Then
	\begin{equation*}
		\begin{split} 
           \mathcal{E}_{2 \gamma}(U) \geq & 2^{1-2\tilde{\gamma}}\frac{\tilde{\gamma}}{1-\tilde{\gamma}}\frac{\Gamma(-\tilde{\gamma})}{\Gamma(\tilde{\gamma})}C^{-1}_{n,2\tilde{\gamma}}\left(\oint_{\H^{n}} |\phi|^{\frac{2m}{m-\tilde{\gamma}}}dzdt\right)^{\frac{m-\tilde{\gamma}}{m}}\\
           &+2^{3-2\gamma}\frac{\Gamma(2-\gamma)}{\Gamma(\gamma)}C^{-1}_{n,2\gamma}\left(\oint_{\H^{n}} |\phi|^{\frac{2m}{m-\gamma}}dzdt\right)^{\frac{m-\gamma}{m}},
          \end{split}
	\end{equation*}
	for all $U \in \mathcal{C}_{D}^{2 \gamma}$, where 				\begin{equation*}
			\psi=B_{2[\gamma]}^{2\gamma}(U),\quad \phi=B_{0}^{2\gamma}(U).
	\end{equation*}
Moreover, equality holds if and only if $U$ is the solution of ~\eqref{Dirichlet2} and 
\begin{equation*}
\begin{split}
      \psi(z,t)&=C_1\left(\frac{1}{(1+|z|^2)^2+t^2}\right)^{\frac{m-\tilde{\gamma}}{2}}, \\
      \phi(z,t)&=C_2\left(\frac{1}{(1+|z|^2)^2+t^2}\right)^{\frac{m-\gamma}{2}}
      \end{split}
\end{equation*}
for some $C_1, C_2\in \C$ and up to CR automorphisms on $\H^n$.
\end{enumerate}
\end{corollary}

\bibliography{mybib}{}
\bibliographystyle{alpha}
		\end{document}